\documentclass[10pt,a4paper,twoside]{amsart}
\usepackage[T1]{fontenc}
\usepackage[utf8]{inputenc}
\usepackage[english]{babel}
\usepackage{amsmath}
\usepackage{amsthm}
\usepackage{amssymb}
\usepackage{amsfonts}
\usepackage{amsxtra}
\usepackage{enumerate}
\usepackage{verbatim}
\usepackage{color}
\usepackage[mathscr]{eucal}

\newcommand{\N}{\mathbb N}
\newcommand{\Z}{\mathbb Z}

\newcommand{\R}{\mathbb R}

\newcommand{\J}{\mathcal{J}}

\newtheorem{theorem}{Theorem}[section]
\newtheorem{definition}[theorem]{Definition}
\newtheorem{lemma}[theorem]{Lemma}
\newtheorem{cor}[theorem]{Corollary}

\newtheorem{prop}[theorem]{Proposition}
\newtheorem*{remark}{Remark}
\newtheorem*{exmp}{Example}

\DeclareMathOperator*{\dist}{dist}

\let\phi=\varphi

\newcommand{\abb}[3]{#1\colon #2\rightarrow #3}
\newcommand{\real}[1]{\mathbb{R}^{#1}}

\begin{document}

\title[Conformally embedded spacetimes]{Conformally embedded spacetimes and the space of null geodesics}

\author{Jakob Hedicke}
\address{Ruhr-Universit\"at Bochum\\ Fakult\"at f\"ur Mathematik\\ Universit\"atsstra\ss e 150\\ 44801 Bochum, Germany}
\email{Jakob.Hedicke@ruhr-uni-bochum.de\\ Stefan.Suhr@ruhr-uni-bochum.de}
\thanks{This research is supported by the SFB/TRR 191 ``Symplectic Structures in Geometry, Algebra and Dynamics'', funded by the Deutsche Forschungsgemeinschaft.}

\author{Stefan Suhr}

\date{\today}

\begin{abstract}
It is shown that the space of null geodesics of a causally simple Lorentzian manifold is Hausdorff if it admits an open conformal embedding into a globally hyperbolic spacetime. This provides an obstruction to conformal embeddings of causally simple spacetimes into 
globally hyperbolic ones irrespective of curvature conditions. Examples of causally simple spacetimes are given not conformally embeddable into globally hyperbolic ones. 
\end{abstract}
\maketitle

\section{Introduction}

Let $(M,g)$ be a spacetime, i.e. a connected time-oriented Lorentzian manifold. The set of null geodesics of $(M,g)$ naturally carries a topology as the 
quotient of the null cones
$\{v\in TM|\; g(v,v)=0,\; v\neq 0\}$ by the actions of the geodesic flow and the Euler vector field. It is shown in \cite{Low90} the space of null geodesics 
$\mathcal{N}_g$ retains a smooth structure from the tangent bundle if the spacetime is strongly causal. In general, though, this 
smooth structure does not induce a manifold structure since the topology might not be Hausdorff. A simple example is given by Minkowski space from 
which one point is deleted. 

Up to this point only two classes of spacetimes were known 
where $\mathcal{N}_g$ is a smooth manifold. On the one hand are the globally hyperbolic spacetimes, for which the space of null geodesics is 
diffeomorphic to the spherical tangent bundle of any Cauchy hypersurface, see \cite{Low}. One the other hand are the {\it Zollfrei} spacetimes, see \cite{guillemin,suhr13}. Zollfrei spacetimes are compact Lorentzian manifolds such that the geodesic flow restricted to the null cones induces an fibration by
circles. The geodesic flow thus projects to a free circle action, which readily implies that the orbit space is a smooth manifold. 

If the space of null geodesics is not Hausdorff, it is shown in \cite{Low90} that the spacetime must admit a {\it naked singularity}, i.e. there exists a PIP 
that contains a TIP (see \cite{FHS} for definitions). In \cite{Low90-2} it is shown that the Hausdorff property of $\mathcal{N}_g$ is equivalent to the {\it null 
pseudoconvexity} of the spacetime. Null pseudoconvexity is a causal condition, which up to this point does not fit into the causal hierarchy, see 
\cite{Sanchez}. In this context it is interesting to determine the precise position in the causal hierarchy.

Motivated by work on the interplay between causal relations in spacetimes and the contact geometry of $\mathcal{N}_g$ Chernov posed in \cite{chernov18} 
two conjectures on causally simple spacetimes and their spaces of null geodesics. More precisely he conjectured that (1) every causally simple spacetime 
admits a conformal embedding into a globally hyperbolic one and (2) if such a conformal embedding exists, the space of null geodesics embeds as an open (contact) 
submanifold. In section \ref{results} below counterexamples to both conjectures are discussed. 

The main purpose of this article though is to give a proof to the weaker formulation of the second conjecture (Theorem \ref{cor1} below) saying that if 
a causally simple spacetime conformally embeds into a globally hyperbolic one, 
the space of null geodesics is Hausdorff, thus showing that in this case the space of null geodesics is a smooth contact manifold. With the richness of examples of such spacetimes one can expect new classes of contact manifolds to 
appear, possibly with exotic contact geometric properties. In the contraposition Theorem \ref{cor1} gives an obstruction to the existence of a conformal 
embedding of a causally simple spacetime into a globally hyperbolic one. The construction in Theorem \ref{T2} provide examples of causally 
simple spacetimes whose space of null geodesics is not Hausdorff and which are therefore not conformally embeddable into a globally hyperbolic spacetime.

\section{Results}\label{results}

Let $(M,g)$ be a spacetime, i.e. a time-oriented Lorentzian manifold. The space of null geodesics $\mathcal{N}_g$ of $(M,g)$ is defined as follows, see 
\cite{Low}: The basic outline is given in the following for the convenience of the reader. The metric $g$ induces a Hamiltonian function 
$$E_g\colon T^*M\to \R,\; \alpha\mapsto g^*(\alpha,\alpha)$$ 
where $g^*$ denotes the dual metric of $g$. The Hamiltonian flow of $E_g$, also called the cogeodesic flow, with respect to the canonical symplectic structure on $T^*M$ is dual to the geodesic flow of $(M,g)$ via the Legendre transform of $g$. Denote with $X_g$ the generator, i.e. the symplectic 
gradient of $E_g$, of the Hamiltonian flow of $E_g$. It is well known that the cogeodesic flow is tangent to the level sets of $E_g$. Thus the {\it future pointing 
dual null cones} 
$$\mathcal{L}^*M:=\{\alpha\in T^*M|\;g^*(\alpha,\alpha)=0,\; \alpha\neq 0, \alpha=g(v,.)\text{ for $v$ future pointing}\}$$
are preserved by the flow. One decisive feature which sets $\mathcal{L}^*M$ apart from the other level sets of $E_g$ is that it is invariant 
under homotheties $\alpha\mapsto t\alpha$ for $t> 0$. The Euler vector field $\xi$ is thus tangent to the dual null cones as well. 

It is easy to see that the commutator of $X_g$ and $\xi$ is co-linear to $X_g$, i.e. by Frobenius' Theorem their span forms an integrable distribution
on the cotangent bundle and by restriction an integrable distribution on $\mathcal{L}^*M$. Denote with $\mathcal{F}_{\text{null}}$ the induced foliation of 
$\mathcal{L}^*M$. By construction a leaf of $\mathcal{F}_{\text{null}}$ consists of the cotangents $g(\dot\gamma(t),.)$ to a null geodesics $\gamma$ and all
its orientation preserving affine reparameterizations. Denote the leaf space of $\mathcal{F}_\text{null}$ with $\mathcal{N}_g$. The leaf space can identified with 
the space of null geodesics that coincide up to affine parametrizations. 
Equip $\mathcal{N}_g$ with the quotient topology relative to $\mathcal{F}_\text{null}$. The quotient topology on $\mathcal{N}_g$ can be characterized via
aa definition of convergence of sequences: One says that the sequence $\{\kappa_n\}_{n\in\N}\subset \mathcal{N}_g$ {\it converges to} $\kappa\in 
\mathcal{N}_g$ if there exist affine parametrizations $\eta_n$ of $\kappa_n$ and $\eta$ of $\kappa$ such that $\dot\eta_n(0)\to \dot\eta(0)$.

If $(M,g)$ is strongly causal every leaf of $\mathcal{F}_\text{null}$
is closed, i.e. the leafs are $2$-dimensional submanifolds of $\mathcal{L}^*M$. In this case the leaf space $\mathcal{N}_g$ inherits a smooth structure from 
$\mathcal{L}^*M$, see \cite[Proposition 11.4.2]{Brickell}. Recall that a smooth structure on a space $\mathcal{M}$ is by definition a maximal atlas of 
homeomorphisms $\phi\colon U^\phi\to V^\phi$, called charts, between open sets $U^\phi\subset \mathcal{M}$ and $V^\phi\subset \R^n$ such that 
every change of chart is a smooth map between open subsets of euclidian space. Thus all notions of calculus are well defined in the case 
of smooth structures as well. 

\begin{prop}\label{propnull}
If the smooth structure of $\mathcal{L}^*M$ descends to $\mathcal{N}_g$, then $\mathcal{N}_g$ inherits a canonical contact structure from the kernel 
of the canonical $1$-form $\theta$ on $T^*M$. 
\end{prop}

\begin{proof}
First note that both $X_g$ and $\xi$ lie in $\ker\theta$ along 
$\mathcal{L}^*M$. Further note that the cogeodesic flow preserves $\theta$ and the flow of the Euler vector field preserves 
$\ker\theta$. Thus the distribution $\ker\theta$ induces a well-defined hyperplane distribution on the quotient of $\mathcal{L}^*M$
by the action of $X_g$ and $\xi$.

Now fix $\alpha\in\mathcal{L}^*M$. Choose tangent vectors $V_1,\ldots,V_{2n-3}\in T\mathcal{L}^*M_\alpha$ such that 
$$\{X_g,\xi,V_1,\ldots,V_{2n-3}\}$$ 
forms a basis of $T\mathcal{L}^*M_\alpha$.
Further choose $W\in TT^*M_\alpha$ with $d\theta(W,\xi)=d\theta(W,V_i)=0$ for all $i=1,\ldots,2n-3$.
Then one has 
\begin{align*}
0\neq &(d\theta)^n(W,X_g,\xi,V_1,\ldots,V_{2n-3})\\
=&d\theta(W,X_g)\cdot (d\theta)^{n-1}(\xi,V_1,\ldots,V_{2n-3})\\
=&-dE_g(W)\cdot \theta\wedge (d\theta)^{n-2}(V_1,\ldots,V_{2n-3}),
\end{align*}
i.e. $\ker\theta\cap T\mathcal{L}^*M$ is a well defined smooth distribution by hyperplanes in $T\mathcal{L}^*M$ which induces a well defined
contact structure on $\mathcal{N}_g$.
\end{proof}

In case the spacetime $(M,g)$ is globally hyperbolic it is well known \cite{Low} that $\mathcal{N}_g$ with the induced contact structure is contactomorphic to 
the unit tangent bundle of any smooth Cauchy hypersurface in $(M,g)$ with its canonical contact structure.

For a spacetime $(\mathcal{M},\mathcal{G})$ denote with $J^+_\mathcal{G}\subset \mathcal{M}\times \mathcal{M}$ the causal relation in 
$(\mathcal{M},\mathcal{G})$. Similarly denote with $I^{+}_\mathcal{G}\subset \mathcal{M}\times \mathcal{M}$ the chronological relation
and the horismo relation of $(\mathcal{M},\mathcal{G})$ with $E^{+}_\mathcal{G}:=J^{+}_\mathcal{G}\setminus I^{+}_\mathcal{G}$, see \cite{Sanchez}.

\begin{prop}\label{prop2}
Let $(M,g)$ be a spacetime such that the space of null geodesics $\mathcal{N}_g$ is Hausdorff. Then the horismos $E_g^+$ is closed.
\end{prop}

\begin{remark}
The fact that $E_g^+(p):=E^+_g\cap \{p\}\times M$ and $E^-_g(p):=E^+_g\cap M\times \{p\}$ are closed for every $p\in M$ does not in general imply that $E^+_g$ is closed as a subset of $M\times M$.
Consider for example the Minkowski space $\mathbb{L}^2$ with a point removed.
Then $E^{\pm}_g(p)$ is closed for every $p$, but there exist sequences $(p_n,q_n)\in E^+_g$ with $(p_n,q_n)\rightarrow (p,q)\in M\times M$ and $(p,q)\notin E^+_g$.
\end{remark}

\begin{proof}[Proof of Proposition \ref{prop2}]
Assume that $\mathcal{N}_g$ is Hausdorff. Let 
$$\{(p_n,q_n)\}_{n\in \N}\subset E_g^+$$ 
be a convergent sequence with limit $(p,q)\in M\times M$. By definition there exists a sequence 
$$\{\gamma_n\colon [0,T_n]\to M\}_{n\in\N}$$ 
of null geodesics connecting $p_n$ with $q_n$. Up to passing to a subsequence one can assume that $\dot\gamma_n(0)$ and $\dot\gamma_n(T_n)$
normalized with respect to a Riemannian metric converge to null vectors $v\in TM_p$ and $w\in TM_q$, respectively. That is equivalent to saying that the 
sequence $\{[\gamma_n]\}_{n\in\N}$ 
converges in $\mathcal{N}_g$ to classes represented by $\gamma_{v}$ and $\gamma_{w}$, where $\dot\gamma_{v}(0):=v$ and $\dot\gamma_w(0):=w$ 
define the geodesics. Since $\mathcal{N}_g$ 
is Hausdorff one concludes $[\gamma_{v}]=[\gamma_{w}]$. Thus the point $q$ lies on $\gamma_{v}$, which shows $(p,q)\in J^+_g$. 
It is obvious that $(p,q)\notin I^+_g$ since otherwise it follows $(p_n,q_n)\in I^+_g$ for sufficiently large $n$ which contradicts the 
assumption $(p_n,q_n)\in E^+_g=J^+_g\setminus I^+_g$. Therefore one has $(p,q)\in E^+_g$. 
\end{proof}

\begin{definition}[\cite{Sanchez}]
A spacetime $(M,g)$ is {\it causally simple} if it is causal and $J^+_g$ is closed.
\end{definition}

\begin{prop}\label{prop1}
Let $(M,g)$ be a simply connected two dimensional spacetime. Then $\mathcal{N}_g$ is a smooth manifold if and only if $(M,g)$ is causally simple. 
\end{prop}

The conformal class of $g$ is defined as 
$$[g]:=\{e^f g|f\in\mathcal{C}^{\infty}(M,\real{})\}.$$
Let $(M,g)$ and $(N,h)$ be smooth Lorentzian manifolds of the same dimension $m\geq 
2$.
Assume that $(M,g)$ embeds conformally as an open subset into $N$, i.e. there exists an open embedding $i\colon M\hookrightarrow N$
such that $i^{\ast}h=e^f g$ for some function $\abb{f}{M}{\real{}}$.

\begin{theorem}\label{cor1}
Let $M,N$ be smooth manifolds of the same dimension. Assume that $(N,h)$ is globally hyperbolic and $(M,g)$ embeds conformally into $(N,h)$.
If $(M,g)$ is causally simple, the space of null geodesics $\mathcal{N}_g$ is a smooth contact manifold.
\end{theorem}

\begin{remark}
If $(M,g)$ conformally embeds into $(N,h)$ the canonical map $\mathcal{N}_g\hookrightarrow \mathcal{N}_h$ is an immersion, provided both spaces have
a smooth structure. Taking this observation into account, Theorem \ref{cor1} confirms a weaker version of \cite[Conjecture 3.7]{chernov18}. 
\end{remark}

\begin{proof}[Proof of Theorem \ref{cor1}]
According to \cite{Low90} the space $\mathcal{N}_g$ inherits a smooth structure if $(M,g)$ is strongly causal.  Any causally simple spacetime is strongly causal, 
see \cite{Sanchez}. The canonical $1$-form on $T^*M$ induces a contact structure on $\mathcal{N}_g$ by Proposition \ref{propnull}. The topology of 
$\mathcal{N}_g$ is Hausdorff by the next proposition. 
\end{proof}

\begin{prop}\label{T1}
Let $M,N$ be smooth manifolds of the same dimension. Assume that $(N,h)$ is globally hyperbolic and $(M,g)$ embeds conformally into $(N,h)$.
If $(M,g)$ is causally simple the space of null geodesics $\mathcal{N}_g$ is Hausdorff.
\end{prop}

\begin{exmp}
If $(M,g)$ embeds conformally into $(N,h)$ the space $\mathcal{N}_g$ does not in general embed into $\mathcal{N}_h$. Consider the two dimensional Minkowski spacetime $I(\overline{N},\overline{h})$ with $\overline{N}:=\R^2$ and the Lorentzian inner product $\overline{h}=dx^2-dy^2$. 
Let $M$ be the open interior of the convex hull of $\{(0,0),(3/2,1/2), (1,1),(1/2,-1/2)\}$. Clearly $(M,dx^2-dy^2)$ is globally hyperbolic, hence causally simple.
Next consider the quotient $N:=\overline{N}/\Z$ where $\Z\times\R^2\to \R^2$, $(k,(x,y))\mapsto (x+k,y)$.
Since the action is isometric for $\overline{h}$, a Lorentzian metric is induced on $N$. This metric is globally hyperbolic as well.
Note that the canonical projection $\overline{N}\to N$ is a diffeomorphism from $M$ onto its image which will be denoted with $M$ as well. 
With this it follows that $M\subset N$ is globally hyperbolic.
The null geodesic in $N$ which lifts to the null geodesic through $(1/4,1/4)$ with direction $(-1,1)$ intersects $M\subset N$ twice.
Therefore the map $\mathcal{N}_g\to \mathcal{N}_h$ induced by the inclusion is not injective. This shows that one cannot expect an embedding of 
$\mathcal{N}_g$ into $\mathcal{N}_h$ even if the $(M,g)$ conformally embeds into $(N,h)$, thus giving a counterexample to \cite[Conjecture 3.7]{chernov18}.
\end{exmp}

Looking at the assumptions of Theorem \ref{cor1} one can wonder if it is necessary to assume the conformal embedding into a globally hyperbolic spacetime 
or if the space of null geodesics for every casually simple spacetime is Hausdorff. The following construction will show that both there are causally simple spacetimes
which do not embed into a globally hyperbolic one and whose space of null geodesics is not Hausdorff. The constructed spacetime thus disproves 
\cite[Conjecture 3.6]{chernov18}.

Consider a smooth function $r\colon \R\to\R$ with $r|_{(0,1)} >0$, $r(0)=r(1)=0$ and $|r'(0)|,|r'(1)|<\frac{1}{2\pi}$. The graph of $r|_{(0,1)}$ defines a surface of revolution 
$\Sigma$ parametrized by 
$$X\colon (0,1)\times \R\to \R^3,\; (x,\phi)\mapsto (x,r(x)\cos\phi,r(x)\sin(\phi)).$$
The induced metric on $\Sigma$ is given by 
$$k=\left[1+(r'(x))^2\right]dx^2+r^2(x)d\phi^2.$$

\begin{theorem}\label{T2}
The spacetime 
$$(M,g)=(\R\times \Sigma,-dt^2+k)$$
is causally simple. Further the space of null geodesics of $(M,g)$ is not Hausdorff and $(M,g)$ does not admit a conformal embedding into a 
globally hyperbolic spacetime.
\end{theorem}

\section{proofs}

\subsection{Proof of Proposition \ref{prop1}}

For orientable $2$-dimensional spacetimes the co-null cones are the union of two transversal $1$-dimensional co-distributions. Thus the space of null geodesics is the union 
of two leaf spaces of two transversal foliations of $M$. 

Assume first that $(M,g)$ is causally simple. By \cite[Proposition 11.4.2]{Brickell} it suffices to show that $\mathcal{N}_g$ is Hausdorff. Since the quotient topology 
on $\mathcal{N}_g$ is second countable it suffices to show that limits of sequences are unique. Let $\kappa^1,\kappa^2\in \mathcal{N}_g$
and $\{\kappa_n\}_{n\in\N}\subset \mathcal{N}_g$ be a sequence with $\kappa_n\to \kappa^1$ and $\kappa_n\to \kappa^2$. Choose parametrizations $\eta^1\colon J_1\to M$,
$\eta^2\colon J_2\to M$ and $\eta_n\colon J_n\to M$ of $\kappa^1$, $\kappa^2$ and $\kappa_n$, respectively and $s\in J_1\cap J_n$ and $t\in J_2\cap J_n$ with 
$\dot\eta_n(s)\to \dot\eta^1(s)$ and $\dot\eta_n(t)\to \dot\eta^2(t)$. By relabelling $\eta^1$ and $\eta^2$ one can assume that $s<t$. 
Since $(M,g)$ is causally simple one has $(\eta^1(s),\eta^2(t))\in J^+_M$. 

It is well known that in $2$-dimensional spacetimes no pair of points is conjugated along a null geodesic. Further since $M$ is simply connected 
and the null geodesics form two transversal foliations of $M$, no null geodesic of $(M,g)$
has cut points. Thus a null geodesic is up to parametrization the unique causal curve connecting any pair of points on it. This shows $\eta^1\equiv \eta^2$, i.e.
$\kappa^1=\kappa^2$. Thus limits in $\mathcal{N}_g$ are unique, i.e. the quotient topology on $\mathcal{N}_g$ is Hausdorff. 

Now assume that $\mathcal{N}_g$ is a smooth manifold. Since every leaf of $\mathcal{F}_\text{null}$ is connected, the manifold $\mathcal{N}_g$ is itself simply connected. 
Thus $\mathcal{N}_g$ is diffeomorphic to the disjoint union of two open intervals $I_1, I_2$. Denote with $f_i\colon M\to I_i$ for $i=1,2$ the canonical maps. Note that both maps
are smooth. By switching the orientation if necessary one can assume that $df_i(v)\ge 0$ for future pointing $v\in TM$ and $i=1,2$. It follows that $v\in TM$ is future pointing 
if and only if $df_1(v)\ge 0$ and $df_2(v)\ge 0$. 

\begin{lemma}
For $p,q\in M$ one has $(p,q)\in J^+_M$ if and only if $f_1(q)\ge f_1(p)$ and $f_2(q)\ge f_2(p)$.
\end{lemma}

\begin{proof}
Assume that $(p,q)\in J^+_M$. Let $\gamma\colon [0,1]\to M$ be a future pointing curve between $p$ and $q$. Then one has 
$$f_i(q)-f_i(p)=\int_0^1 df_i(\dot\gamma(t))dt \ge 0$$ 
for $i=1,2$. 

Now assume $f_1(q)\ge f_1(p)$ and $f_2(q)\ge f_2(p)$. Let $\alpha\colon [0,1]\to M$ be a curve between $p$ and $q$. Set 
$$t_1:=\max\{t\in [0,1]|\; f_1\circ\alpha (t)=f_1(p)\text{ or }f_2\circ \alpha (t)=f_2(p)\}.$$
There exists null geodesic $\beta_1\colon [0,t_1]\to M$ between $p$ and $\alpha(t_1)$ since the leafs of the null foliations are connected. By the intermediate value theorem
one has $f_i(\alpha(t_1))\ge f_i(p)$ for $i=1,2$, i.e. $\beta_1$ is future pointing. Replace $\alpha|_{[0,t_1]}$ with $\beta_1$. For the resulting $\alpha_1\colon [0,1]\to M$ one 
has $f_i\circ \alpha_1\ge f_i(p)$ for $i=1,2$. Next let 
$$t_2:=\min\{t\in [0,1]|\; f_1\circ\alpha_1 (t)=f_1(q)\text{ or }f_2\circ \alpha_1 (t)=f_2(q)\}.$$
The null geodesic $\beta_2\colon [t_2,1]\to M$ between $\alpha_1(t_2)$ and $q$ exists and is future pointing by the same argument as before.
Replace $\alpha_1|_{[t_2,1]}$ with $\beta_2$. The resulting curve $\alpha_2\colon [0,1]\to M$ satisfies $f_i(p)\le f_i\circ \alpha_2\le f_i(q)$ for $i=1,2$. 

If $f_1\circ \alpha_2$ or $f_2\circ \alpha_2$ is constant, then $p$ and $q$ lie on a common null geodesic. By the assumptions it  follows that $(p,q)\in J^+_M$.
One can thus assume that both $f_1\circ \alpha_2$ and $f_2\circ \alpha_2$ are non-constant. Perturb $\alpha_2$ to a smooth curve $\alpha_3\colon [0,1]\to M$ 
between $p$ and $q$ with $f_i(p)\le f_i\circ \alpha_3\le f_i(q)$ for $i=1,2$ and such that $f_1\circ \alpha_3$ has only non-degenerate critical points. 
Choose a local minimum of $f_1\circ \alpha_3$ and a parameter $s\in [0,1]$ where it is attained. Let
$$r:=\min\{r'|\; r'<s,\, f_1\circ \alpha_3(r')=f_1\circ \alpha_3(s)\}.$$
Replace $\alpha_3|_{[r,s]}$ with the future pointing null geodesic between the endpoints. Continue inductively over the set of local minima attained outside
of $[r,s]$. 
For the obtained Lipschitz continuous curve $\alpha_4\colon[0,1]\to M$ all left- and right-sided differentials $\dot\alpha_4^\pm$ exists and one has 
$df_1(\dot\alpha_4^\pm)\ge 0$.

Perturb $\alpha_4$ to a smooth curve $\alpha_5$ such that $df_1(\dot\alpha_5)\ge 0$ and $f_2\circ \alpha_5$ has only non-degenerate critical points. 
One distinguishes two cases: First, if $f_1\circ \alpha_5\equiv \text{const}$ the curve $\alpha_5$ can be reparametrized to a  future pointing null 
geodesic, thus showing $q\in\J^+_M(p)$. Second, if $df_1(\dot\alpha_5)> 0$ somewhere, one can perturb $\alpha_5$ such that $df_1(\dot\alpha_5)> 0$
everywhere. Choose a local minimum of $f_2\circ \alpha_5$ and a parameter $u\in [0,1]$ where it is attained and repeat the induction as in the last 
paragraph. 
For the obtained Lipschitz continuous curve $\alpha_6\colon[0,1]\to M$ all left and right sided differentials $\dot\alpha_6^\pm$ exists and one has 
$df_1(\dot\alpha_6^\pm),df_2(\dot\alpha_6^\pm)\ge 0$. Thus all left and right sided derivatives are future pointing, i.e. $\alpha_6$ is a future pointing curve connecting $p$ and $q$.
\end{proof}

\subsection{Proof of Proposition \ref{T1}}

\begin{lemma}\label{L3}
Let $(\mathcal{M},\mathcal{G})$ be a spacetime such that $E_{\mathcal{G}}^+$ is closed and non-empty. Consider a sequence 
$$\{(p_n,q_n)\}_{n\in\N}\subset E_{\mathcal{G}}^+$$
converging to $(p,q)\in E_{\mathcal{G}}^+$ and a sequence 
$$\{\abb{\eta_n}{[0,b_n]}{\mathcal{M}}\}_{n\in\N}$$
of null geodesics with $\eta_n(0)=p_n$ and $\eta_n(b_n)=q_n$. Then there exists a null geodesic $\eta$ 
connecting $p$ and $q$ such that up to a subsequence $[\eta_n]\rightarrow [\eta]\in\mathcal{N}_{\mathcal{G}}$.
\end{lemma}

\begin{proof}
Choose a complete Riemannian metric on $\mathcal{M}$. The following properties hold up to a subsequence of $\{\eta_n\}_{n\in\N}$ due to the limit curve 
theorem, see \cite{Minguzzi} or \cite{BS1}: Let $\abb{\eta^R_n}{[0,c_n]}{\mathcal{M}}$ be a Riemannian arclength parametrisation of $\eta_n$. The
sequence $\{\eta^R_n\}_{n\in\N}$ converges uniformly on compact subsets with respect to the Riemannian metric to a causal curve $\abb{\eta^R}{[0,c)}{\mathcal{M}}$ with 
$\eta^R(0)=p$. Since the $\eta_n^R$'s are null pregeodesics so will be $\eta^R$. 

If $c<\infty$ one concludes that $\eta^R$ extends uniquely to $c$ with $\eta^R(c)=q$. Let $\eta\colon[0,b]\to \mathcal{M}$ be an affine 
parameterisation of $\eta^R$. It follows that $\eta$ is a null geodesic between $p$ and $q$ with $[\eta_n]\rightarrow [\eta]
\in\mathcal{N}_{\mathcal{G}}$.

Otherwise one has $c=\infty$ and $\eta^R$ is future inextensible. 
Choose $0<s<t<\infty$ and $n$ sufficiently large such that $t<c_n$.  By a standard argument one has 
$$(\eta_n^R(s),q_n),(\eta_n^R(t),q_n), (\eta^R_n(s),\eta^R_n(t))\in  E_{\mathcal{G}}^+.$$
The assumptions that $\eta_n^R(s)\rightarrow \eta^R(s)$ and $\eta_n^R(t)\rightarrow \eta^R(t)$ as well as 
that $E_{\mathcal{G}}^+$ is closed imply
$$(\eta^R(s),q),(\eta^R(t),q),(\eta^R(s),\eta^R(t))\in E_{\mathcal{G}}^+.$$
The up to parametrisation unique causal curve connecting 
$\eta^R(s)$ and $q$ has to be $\eta^R$ since $\eta^R$ is the unique causal curve between $\eta^R(s)$ and $\eta^R(t)$.
Otherwise this would imply $q\in I^+_\mathcal{G}(\eta^R(s))$. This contradicts the future inextensibility of $\eta^R$.
\end{proof}

Let $C_h^+\subset E_h^+$ denote the future null cut locus in $N$, i.e. $(p,q)\in C_h^+$ if $(p,q)\in E_h^+$ and there exists a null geodesic from 
$p$ to $q$ that stops being unique at $q$, see \cite[Chapter 9]{Beem}. 
Let 
$$L_g^+:=E_g^+\cap( E_h^+\setminus C_h^+).$$ 
Obviously is $L^+_g$ the set of pairs of points in $M$ connected by future directed causal curve in $(N,h)$ and $(M,g)$ unique up to parametrisation.

\begin{lemma}\label{L1}
Let $M\subset N$ be open, $h$ a globally hyperbolic Lorentzian metric on $N$ such that $(M,g):=(M,h|_M)$ is causally simple. 
Then the set $L_g^+$ is a connected component of 
$$(E_h^+\setminus C_h^+)\cap (M\times M).$$
\end{lemma}

\begin{proof}
Since $(M,g)$ is causally simple the set $E_g^+$ is closed in $M\times M$.
Hence 
$$L_g^+=E_g^+\cap ((E_h^+\setminus C_h^+)\cap (M\times M))$$ 
is closed in $(E_h^+\setminus C_h^+)\cap (M\times M)$ in the subspace topology.

To show that it is also open, assume that the complement of $L_g^+$, 
$$((E_h^+\setminus C_h^+)\cap (M\times M))\setminus L_g^+,$$ 
is not closed. Then there exists a sequence $(p_n,q_n)\rightarrow (p,q)$ with 
$$(p_n,q_n)\in ((E_h^+\setminus C_h^+)\cap (M\times M))\setminus L_g^+\text{ and }(p,q)\in L_g^+.$$
Thus there exist null geodesics $\eta_n\colon[a_n,b_n]\to N$ connecting $p_n$ and $q_n$ converging due to 
Lemma \ref{L3} to the unique null geodesic $\eta\colon[a,b]\to M$ connecting $p$ and $q$. Note that the $\eta_n$'s are unique up to parametrization 
since $(p_n,q_n)\in E^+_h\setminus C^+_h$ for all $n\in\N$. The curves are further not contained in $M$ since this would imply $(p_n,q_n)\in E^+_g$, 
which contradicts the assumption $(p_n,q_n)\notin L^+_g$.
Hence there exist points $\eta_n(t_n)\in N\setminus M$.
A subsequence of $\eta_n(t_n)$ converges to a point $\eta(t)\in M$. This contradicts $M$ being open in $N$. Hence $L_g^+$ is also open in 
$(E_h^+\setminus C_h^+)\cap (M\times M)$. This shows that $L^+_g$ is a union of connected components of $(E_h^+\setminus C_h^+)\cap (M\times M)$.

It remains to show that $L_g^+$ is path-connected. Since $h$ is globally hyperbolic the set $E^+_h\setminus C^+_h$ contains the diagonal in $N\times N$. 
Since $M$ is an open subset of $N$ this implies that $E^+_g$ contains the diagonal of $M\times M$. Further for $(p,q)\in E^+_h\setminus C^+_h$ and 
$\eta\colon [0,1]\to N$ a casual curve from $p$ to $q$ one has $(p,\eta(t))\in E^+_h\setminus C^+_h$ for all $t\in[0,1]$. The same goes for 
$(p,q)\in E^+_M$ and any causal curve connecting the points in $M$. This shows that $L^+_g$ is path-connected.
\end{proof}

\begin{lemma}\label{lemma_connect}
Let $(N,h)$ be a globally hyperbolic spacetime.
Let $\eta\colon I\to N$ be an inextensible null geodesic and $s,u\in I$ with $s<u$ such that $\eta(u)\in E^+_{h}(\eta(s))$. 
Then for all $t\in [s,u)$ there exists $r\in I$ with $r<s$ such that 
$\eta|_{[r,t]}$ is up to parametrization the unique causal curve between $\eta(r)$ and $\eta(t)$.
\end{lemma}

\begin{proof}
Let the open set $\mathbb{U}\subset\R\times TN$ be the maximal domain of the geodesic flow $\abb{\Phi^h}{\mathbb{U}}{TN}$ of 
$h$.
Recall that there exists a neighbourhood of the zero section in $TN$ such that $\{1\}\times U\subset \mathbb{U}$.
Consider 
$$\mathbb{V}:=\{v\in TN|\{1\}\times \{v\}\in\mathbb{U}\}$$
and define the exponential map of $(N,h)$ as 
$$\mathrm{Exp}\colon\mathbb{V}\to N\times N, v\mapsto (\pi_{TN}(v),\pi_{TN}\circ\Phi^h(1,v)),$$
where $\abb{\pi_{TN}}{TN}{N}$ denotes the canonical projection.
Then the exponential map at a point $p\in N$ is defined as
$$\abb{\exp_p}{\mathbb{V}\cap T_pN}{N}, v \mapsto \pi_{TN}\circ\Phi^h(1,v).$$
Both $\mathrm{Exp}$ and $\exp_p$ are smooth and $\abb{d\,\mathrm{Exp}_v}{T_vTN}{T_{(p,\exp_p(v))}(N\times N)}$ is non-degenerate if and only if 
$\abb{d(\exp_p)_v}{T_v(TN_p)}{T_{\exp_p(v)}N}$ is non-degenerate where $p:=\pi_{TN}(v)$, see e.g. \cite{Lee}.

If $\eta(u)\in E^+_{h}(\eta(s))$ for some $u>s$, then $\eta|_{[s,\xi]}$ is the unique causal curve in $N$
between $\eta(s)$ and $\eta(\xi)$ for all $s<\xi<u$.
Since $\eta$ is the unique causal geodesic between $\eta(s)$ and $\eta(u)$ no $\eta(\xi)$ is conjugate to $\eta(s)$ 
along $\eta$ for $\xi\in (s,u)$, see \cite{Beem}.
This yields $d(\exp_{\eta(s)})_{(\xi-s)\dot{\eta}(s)}$ is non-degenerate.
By the above equivalence this implies that $d\,\mathrm{Exp}$ is non-degenerate at $v:=(\xi-s)\dot{\eta}(s)$.
With the implicit function theorem one knows that $\mathrm{Exp}$ is a local diffeomorphism from a neighborhood $U_v$ of $v$ in $TN$ onto a neighborhood $V_v$ of 
$(p,\exp_p(v))$ in $N\times N$. Therefore $\eta|_{[\nu,\xi]}$ is the unique geodesic between $\eta(\nu)$ and $\eta(\xi)$ for $\nu$ sufficiently close to 
$s$ in a neighborhood of $\eta|_{[s,u]}$. The curve $\eta|_{[\nu,\xi]}$ is in fact the unique causal geodesic between its endpoints 
in $N$ for $\nu$ close to $s$. Indeed assume that there exists a causal geodesic $\abb{\tilde{\eta}}{[\nu,\xi]}{N}$ different from $\eta|_{[\nu,\xi]}$ with 
$\tilde{\eta}(\nu)=\eta(\nu)$ and $\tilde{\eta}(\xi)=\eta(\xi)$. Since $\tilde{\eta}$ has to leave a neighborhood of $\eta|_{[s,u]}$ every limit curve of $\tilde{\eta}$ 
for $\nu\rightarrow s$ is a causal geodesic between $\eta(s)$ and $\eta(\xi)$ different from $\eta|_{[s,\xi]}$.
This contradicts the assumption that $\eta(u)\in E_{h}^+(\eta(s))$.
The limit geodesic exists by the limit curve theorem in \cite{Minguzzi,BS1} and the assumption that $N$ is globally hyperbolic.
\end{proof}

\begin{lemma}\label{L41}
Let $N$ be a smooth manifold of dimension at least $3$, $M\subset N$ open and $(N,h)$ globally hyperbolic such that $(M,h|_M)$ is causally simple. 
Further let $\eta\colon [-1,1]\to N$ be a null geodesic with $\eta|_{[-1,0)}\subset M$ and $\eta_n\colon [-1,1]\to M$ be a sequence of null geodesics
with $\dot\eta_n(0)\to \dot \eta(0)$. Assume 
$$[u_n,1]:=\eta_n^{-1}(J^+_h(\eta(0)))\neq \emptyset$$ 
for infinitely many $n\in\N$ and $\liminf u_n=0$. Then 
one has $\eta(0)\in M$.
\end{lemma}

\begin{proof}
Choose an $h$-convex neighborhood $V$ of $\eta(0)$. Fix $\rho\in [-1,0)$ such that 
$\eta(\rho)\in V$ and an $h|_M$-convex neighborhood $W\subset M$ of $\eta(\rho)$.

Let $\tau\colon N\to\R$ be a smooth temporal function. 
By diminishing $V$ and $W$ one can assume that the intersection of both $E^+_h(p)$ and $E^-_h(p)$ with 
$$W_\rho:=W\cap\{\tau= \tau(\eta(\rho))\}$$ 
is path connected for all $p\in V$. Further one can assume that $\tau(\eta_n(\rho))=\tau(\eta(\rho))$ for all $n$.
If $\eta(0)\in \eta_n$ for some $n$ the claim is trivial. Thus one can assume $\eta_n(u_n)\in J^+_h(\eta(0))\setminus \{\eta(0)\}$ for all $n\in\N$.
Then the assumption on $\eta_n$ implies that one can find $\upsilon\in(0,1)$ such that
$$\eta_n(\upsilon)\in I^+_h(\eta(0))\cap V$$
for infinitely many $n$. Choose a compact neighborhood $W'_\rho$ of $\eta(\rho)$ in $W_\rho$ such that there exists 
$$p\in E^-_h(\eta(\upsilon))\cap (W_\rho\setminus W'_\rho)$$ 
and $p_n\in E^-_h(\eta_n(\upsilon))\cap W_\rho$ with $p_n\to p$.
The unique geodesic segment between $p_n$ and $\eta_n(\upsilon)$ belongs to $M$ by Lemma \ref{L1}, since one can find a path in 
$E_h^+\setminus C^+_h\cap (M\times M)$ between $(\eta_n(\rho),\eta_n(\upsilon))$ and $(p_n,\eta_n(\upsilon))$. 
The existence of such a path follows 
from the fact that $W_\rho$ is chosen such that $W_\rho\cap E_h^-(\eta_n(\upsilon))$ is path-connected.
The intersection $y_n$ of the geodesic segment between $p_n$ and $\eta_n(\upsilon)$ with $E^+_h(\eta(0))$ 
converges to $\eta(\upsilon)$ because the intersection of the geodesic segment between $p$ and $\eta(\upsilon)$ with 
$E^+_h(\eta(0))$ is $\eta(\upsilon)$. The unique geodesic in $V$ between $\eta(0)$ and $y_n$ intersects 
$W_\rho$ to the past in a point $x_n$ since $y_n\to \eta(\upsilon)$ and the unique geodesic between $\eta(0)$ and $\eta(\upsilon)$ 
is $\eta$. Like before Lemma \ref{L1} implies $x_n\in E^-_g(y_n)$ using a path between $p_n$ and $x_n$ in $W_\rho\cap E_h^-(y_n)\setminus 
C^-_h(y_n)$. Therefore the geodesic segment between $x_n$ and $y_n$ lies in $M$. It follows that $\eta(0)\in M$. 
\end{proof}

\begin{lemma}\label{L2}
Let $N$ be a smooth manifold of dimension at least $3$, $M\subset N$ open and $(N,h)$ globally hyperbolic such that $(M,h|_M)$ is causally simple. 
Further let $\eta\colon [-1,1]\to N$ be a null geodesic with $\eta|_{[-1,0)}\subset M$ and $\eta_n\colon [-1,1]\to M$ be a sequence of null geodesics
with $\dot\eta_n(0)\to \dot\eta(0)$. Assume that all $\eta_n$ are disjoint from $J^+_h(\eta(0))$. Then there exists a neighborhood $U$ of $\eta(0)$ in $N$ 
such that 
$$E^-_h(\eta(0))\cap U\setminus \{\eta(0)\}\subset M.$$
\end{lemma}

\begin{proof}
Choose an $h$-convex neighborhood $V$ of $\eta(0)$. Fix $\rho\in [-1,0)$ such that 
$\eta(\rho)\in V$ and an $h|_M$-convex neighborhood $W$ of $\eta(\rho)$ in $M$.
Let $\tau\colon N\to\R$ be a smooth temporal function. 
By diminishing $V$ and $W$ one can assume that the intersection of both $E^+_h(p)$ and $E^-_h(p)$ with 
$$W_\rho:=W\cap\{\tau= \tau(\eta(\rho))\}$$ 
is path connected for all $p\in V$.

Choose $\upsilon\in(0,1)$ such that $\eta(\upsilon)\in V$.
Let $\beta\colon [0,1]\to V$ be a future pointing null geodesic segment with $\beta(1)=\eta(0)$ not parallel to $\eta$.
Thus $\beta|_{[0,1)}\subset I^-_h(\eta(\upsilon))$. 
Hence for every $t<1$ there exists $n_0$ such that for all $n\ge n_0$ one has 
$$\beta(t)\in I^-_h(\eta_n(\upsilon)).$$
If $\beta$ is parallel to any $\eta_n$, then $\eta(0)$ lies on $\eta_n$, hence in $M$. 
Then the claim is trivial since $M$ is an open subset of $N$.

Therefore one can assume that $\beta$ is not parallel to any $\eta_n$. This then holds 
for all null geodesics through $\eta(0)$ sufficiently close to $\beta$. 
Since $\eta(0)\notin J^-_h(\eta_n(\upsilon))$ there exists $t_n\in (t,1)$ such that $\beta(t_n)\in E^-_h(\eta_n(\upsilon))$ for sufficiently large $n$. 
Let $[\gamma_{n,\beta}]\in \mathcal{N}_h$ be the unique class of null geodesics whose representatives contain $\beta(t_n)$ and 
$\eta_n(\upsilon)$. Every representative of $[\gamma_{n,\beta}]$ intersects $W_\rho$ for $n$ sufficiently large since $t_n\to 1$ and 
$\eta_n(\upsilon)\to \eta(\upsilon)$. Thus Lemma \ref{L1} implies that $\beta(t_n)\in M$:
Let $x_n:=\gamma_{n,\beta}\cap W_\rho$. Then one can find a path in $W_\rho\cap E^-_h(\eta_n(\upsilon))$ from $\eta_n(\rho)$ 
to $x_n$ and hence a path in $(E_h^+\setminus C^+_h)\cap (M\times M)$ 
 from $(\eta_n(\rho),\eta_n(\upsilon))$ to $(x_n,\eta_n(\upsilon))$. Since $\gamma_{n,\beta}$ is a causal curve between $x_n$ and $\eta_n(\upsilon)$ 
 unique up to parametrization it follows that $\beta(t_n)\in M$. 

For all $\beta$ with $\beta(0)\in W_\rho$ one has $\beta|_{[0,1)}\subset M$: The claim is trivial for $\beta=\eta$. Therefore assume that $\beta$ is
not parallel to $\eta$. The sub arc of $[\gamma_{n,\beta}]$ between $x_n$ and $\eta_n(\upsilon)$ lies in $M$. Now for every $t_n$ one can chose a path in 
$W_\rho\cap E_h^-(\beta(t_n))$ from $x_n$ to $\beta(0)$. Therefore by Lemma \ref{L1} and local uniqueness of geodesics the geodesic 
arc $\beta|_{[0,t_n]}$ lies in $M$ and since $t_n\rightarrow 1$ this implies $\beta|_{[0,1)}\subset M$.

For geodesics $\beta$ which do not intersect $W_\rho$ to the past let $\beta(t_1)$ and $\beta(t_2)$ be intersections of 
$\beta$ with $E^-_h(\eta_{n_1}(\upsilon))$ and $E^-_h(\eta_{n_2}(\upsilon))$, respectively. Assume $t_1<t_2$ and 
$n_1,n_2$ sufficiently large. Choose $[\gamma_i]:=[\gamma_{n_i,\beta}]\in \mathcal{N}_h$ and $x_i:=x_{n_i}\in W_\rho$ as before.
One has $x_i\in J^-_h(\beta(t_2))$.

The set $J^-_h(\eta_{n_2}(\upsilon))\cap V\setminus \gamma_2$ is foliated by past-pointing null geodesics emanating from points on $\gamma_2$ 
prior to $\eta_{n_2}(\upsilon)$. Consequently a path in $W_\rho\cap J^-_h(\beta(t_2))$ from $x_2$ to $x_1$ joined with 
$\gamma_1|_{[\gamma_1^{-1}(x_1),\gamma_1^{-1}(\beta(t_2))]}$ induce a path in $(E^+_h\setminus C^+_h)\cap (M\times M)$ from $(x_2,x_2)\in L^+_g$ 
to $(\beta(t_1),\beta(t_2))$, i.e. $(\beta(t_1),\beta(t_2))\in J^+_g$. As before this implies $\beta|_{[t_1,t_2]}\subset M$.
\end{proof}

\begin{lemma}\label{L4}
Let $N$ be a smooth manifold of dimension at least $3$, $M\subset N$ open and $(N,h)$ globally hyperbolic such that $(M,h|_M)$ is causally simple. 
Further let 
$$\eta\colon [-1,1]\to N$$ 
be a null geodesic with $\eta|_{[-1,0)}\subset M$ and 
$$\{\eta_n\colon [-1,1]\to M\}_{n\in \N}$$ 
be a sequence of null geodesics with $\dot\eta_n(0)\to\dot\eta(0)$. Then the set $\eta^{-1}(M)$ is connected.
\end{lemma}

\begin{proof}
Let $r,w\in \eta^{-1}(M)$ with $r<w$ and assume that there exists $r<s <w$ with $\eta(s)\notin M$. Without loss of generality one can assume that
$s$ is minimal in that respect, i.e.
$$s=\inf\{s'>r|\,\eta(s')\notin M\}.$$ 
By Lemma \ref{L41} one knows that the sets $\eta_n^{-1}(J^+_h(\eta(s)))$ are bounded away from $s$. 
Fix $n$ sufficiently large such that there exists $p\in I_g^+(\eta_n(w))\cap I_g^+(\eta(w))$.
Choose a timelike curve $\alpha\colon[w,2]\to M$ from $\eta_n(w)$ to $p$.
Define
$$\abb{\beta}{[-1,2]}{M},\; \beta(t):=\left\lbrace\begin{array}{cc}
\eta_n(t), & t\leq w \\ 
\alpha(t), & t>w.
\end{array}\right.$$

Let $\tau \colon N\to \R$ be a smooth temporal function. Set $\sigma:=\tau(\eta(s))$ and choose a compact neighborhood $U$ of $\eta(s)$ according to 
Lemma \ref{L2} such that $U\cap E_h^-(\eta(s))\setminus \{\eta(s)\}\subset M$. Note that by the assumption that $(M,h|_M)$ is causally simple it follows that 
$\eta(w)\in J^+_g(x)$ for all $x\in U\cap E_h^-(\eta(s))\setminus \{\eta(s)\}$. For $\delta>0$ define 
$$V_\delta :=\tau^{-1}([\sigma-\delta,\sigma))\cap U\cap E_h^-(\eta(s))\subset M$$
and 
$$u_{\delta}:=\inf\{u\in \R|\,\beta(u)\in J_g^+(V_\delta)\}.$$
It follows that the parameter $u_{\delta}$ is bounded from above by $2$ and the function $\delta\mapsto u_{\delta}$ is monotonously decreasing.

For $0<\delta'<\delta$ sufficiently small the set $V_\delta\setminus V_{\delta'}$ is precompact in $M$. Since $(M,g)$ is causally simple the 
precompactness of $V_\delta\setminus V_{\delta'}$ and the monotonicity of $u_\delta$ imply that there exists $x_{\delta}\in V_\delta$ 
such that $\beta(u_{\delta})\in J_g^+(x_{\delta})$.
Furthermore $\beta(u_{\delta})\in E_g^+(x_{\delta})$  by minimality of $u_\delta$.

Take a sequence $\delta_k\downarrow 0$ and a sequence $x_k:=x_{\delta_k}$. By construction one has $x_k\rightarrow \eta(s)$.
Let $\gamma_{k}\colon [0,1]\to M$ be a sequence of null geodesics connecting $x_{k}$ and $\beta(u_{\delta_k})$.
Note  that the sequence $\{u_{\delta_k}\}_{k\in\N}$ is monotonously increasing. 
Since the geodesic flow is smooth one can assume that up to a subsequence the geodesics $\gamma_{k}$ converge in every 
$\mathcal{C}^l$-norm to a null geodesic $\gamma\colon[0,1]\to N$. This geodesic connects $\eta(s)$ and 
$\beta(u_\infty)$, where $u_\infty$ denotes the limit of the sequence $\{u_{\delta_k}\}_{k\in\N}$.

Since $(\gamma_{k}(0),\gamma_{k}(1))\in E_g^+$ one concludes that the index form of every $\gamma_k$ is negative semidefinite (see Appendix \ref{A1}).
Furthermore the negative semi-definiteness is preserved under convergence of geodesics, i.e. the index form of $\gamma$ is negative semi-definite. 
This implies that the index form of $\gamma|_{[0,b]}$ is negative definite for all $b\in (0,1)$. Hence no $\gamma(b)$ is conjugated to 
$\gamma(0)=\eta(s)$ along $\gamma$. Fix $b<1$ such that $\gamma(b)\in M$.

Due to Lemma \ref{L2} one can choose $a<0$ such that $\gamma$ can be extended until $a$ and $\gamma(t)\in M$ for all $t\in [a,0)$.
Since the index form depends continuously on the geodesic one can choose $a$ such that the index form of $\gamma|_{[a,b]}$ is negative 
definite, i.e. no point $\gamma(t)$ is conjugated to $\gamma(a)$ along $\gamma$ for $t\in [a,b]$. Thus $d\exp_{\gamma(a)}$ is non-singular along the line 
$[0,b-a]\cdot\dot\gamma(a)$, i.e. $d\text{Exp}$ is non-singular along the line $[0,b-a]\cdot\dot\gamma(a)$. Thus there 
exists a neighborhood $\mathcal{U}$ of $[0,b-a]\cdot\dot\gamma(a)$ such that 
$$\text{Exp}\colon \mathcal{U}\to \text{Exp}(\mathcal{U})\subset N\times N$$
is a local diffeomorphism. Since $\text{Exp}|_{[0,b-a]\cdot\dot\gamma(a)}$ is bijective one can assume by shrinking $\mathcal{U}$ if necessary that
$\text{Exp}|_\mathcal{U}$ is bijective and $\mathcal{U}$ is fibrewise star-shaped. For $p\in \pi_{TN}(\mathcal{U})$ define $\mathcal{V}_p:=\exp_p(\mathcal{U}\cap TM_p)$.
Let $v\in \mathcal{U}\cap TM_p$ be null. Then \cite[Proposition 5.34]{oneill} implies that there does not exist a timelike curve inside $\mathcal{V}_p$ between $p$ and $\exp_p(v)$.
Applying this to $v=\dot\gamma_k(a)\in \mathcal{U}$ there exists an open neighborhood $\mathcal{V}$ around $\gamma|_{[a,b]}$ such that for sufficiently large $k$ 
the geodesics $\gamma_{k}|_{[a,b]}$ lie inside $\mathcal{V}$ and two points that lie on a $\gamma_{k}|_{[a,b]}$ cannot be connected by a timelike curve inside $\mathcal{V}$.

The claim is now that there exists $a'\in [a,0)$ such that $(\gamma(a'),\gamma(b))\in E^+_g$. Assuming the claim one has $(\gamma(a''),\gamma(b))\in E^+_g$ for all $a''\in [a',0)$. By a 
standard argument it follows that $\gamma|_{[a',b]}\subset M$, especially implying that $\gamma(0)=\eta(s)\in M$. This contradicts the assumption.

The claim is proved if there exists $a'\in [a,0)$ such that $(\gamma_k(a'),\gamma_k(b))\in E^+_g$ for all $k$ sufficiently large. 
Suppose the claim is false, i.e. there exists a sequence $a'_k\uparrow 0$ with $(\gamma_k(a'_k),\gamma_k(b))\in I^+_g$. Choose a 
compact neighborhood $W$ of $\gamma(b)$ in $M\cap \mathcal{V}$. Then for sufficiently large $k$ there exists a null geodesic $\zeta_k\colon [0,1]\to M$ from 
$\gamma_k(a'_k)$ to a point $z_k\in J^-_g(\gamma_k(b))\cap W\setminus \gamma_k$. This follows from the minimality of $u_{\delta_k}$ and $u_{\delta_k}\uparrow u_\infty$.
The geodesics $\zeta_k$ cannot be contained in the neighbourhood $\mathcal{V}$ defined in the previous paragraph, since otherwise $\gamma_k(a)$ and $\gamma_k(b)$ would 
be connected by a timelike curve inside $\mathcal{V}$. A subsequence of $\{\zeta_k\}_{k\in\N}$ converges to a null geodesic $\zeta\colon [0,1]\to N$ connecting $\gamma(0)$ with 
$\gamma(b)$. The second assertion follows since $z_k\to \gamma(b)$ again by the minimality of $u_{\delta_k}$ and $u_{\delta_k}\uparrow u_\infty$.
The geodesic $\zeta$ is not a reparameterization of $\gamma$, i.e. $\dot\zeta(1)$ and $\dot\gamma(b)$ are not parallel. Choose $c<1$ with $\zeta(c)\in W$. Then there
exists $\epsilon >0$ such that $\beta(u_\infty-\epsilon)\in J^+_g(\zeta(c))$. By continuity it follows that $\beta(u_\infty-\epsilon /2)\in J^+_g(\zeta_k(c))\subset J^+_g(\gamma_k(a'_k))$.
Note that for all $\delta >0$ one has either $\gamma_k(a'_k)\in J^+_g(V_\delta)$ or $\gamma_k$ intersects $V_\delta$ for all sufficiently large $k$.
Since $a'_k\uparrow 0$ it follows that $u_\infty <u_{\delta_k}-\epsilon/2$ for all $k$ sufficiently large. This contradicts $\lim_{k\to \infty} u_{\delta_k}=u_\infty$ 
and finishes the proof.
\end{proof}

\begin{proof}[Proof of Proposition \ref{T1}]
Since the null geodesics of two conformal metrics coincide up to reparametrisations, one can assume without loss of 
generality that 
$$i\colon (M,g)\hookrightarrow(N,h)$$ 
is an isometric embedding; in other words $M$ is an open subset of $N$ and $g=h|_M$.

Recall that $\mathcal{N}_g$ was constructed as the quotient of the bundle of null covectors by the action of the Euler vector field and the geodesic flow.
By definition the quotient map is open. Hence the obtained topology on $\mathcal{N}_g$ is second countable and the Hausdorff property is equivalent to the 
uniqueness of limits for converging sequences, see \cite[Proposition 6.5]{Quer}.

Let $[\eta^0], [\eta^1]\in\mathcal{N}_g$
be classes of null geodesics and 
$$\{[\eta_n]\}_{n\in \N}\subset \mathcal{N}_g$$ 
be a sequence such that 
$$\dot{\eta}_n\rightarrow \dot{\eta}^0\text{ and }\dot{\eta}_n\to \dot\eta^1$$
somewhere.
Up to relabelling and reparameterization one can assume that all geodesics are future pointing and $\eta^1\subset J^+_h(\eta^0)$.
From the fact that $\mathcal{N}_h$ is Hausdorff it readily follows that $\eta^0$ and $\eta^1$ are subarcs of the same null geodesic 
$H\colon (\alpha,\omega)\to N$ in $(N,h)$.

By Lemma \ref{L4} the set $H^{-1}(M)$ is connected. Therefore the geodesics $\eta^0$ and $\eta^1$ are subarcs of the same geodesic in $M$, i.e. 
the sequence $\{[\eta_n]\}_{n\in\N}$ has a unique limit in $\mathcal{N}_g$.
\end{proof}

\section{Proof of Theorem \ref{T2}}

Recall that one considers a smooth function $r\colon \R\to\R$ with $r_{(0,1)}>0$, $r(0)=r(1)=0$ and $|r'(0)|,|r'(1)|<\frac{1}{2\pi}$. The graph of $r_{(0,1)}$ defines a surface of 
revolution $\Sigma$ parametrized by 
$$X\colon (0,1)\times \R\to \R^3,\; (x,\phi)\mapsto (x,r(x)\cos\phi,r(x)\sin\phi).$$
The induced metric on $\Sigma$ is given by 
$$k=\left[1+(r')^2\right]dx^2+r^2d\phi.$$

\begin{lemma}\label{L10}
Every geodesic of $(\Sigma,k)$ is either complete or is asymptotic in both direction to $x=0$ or $x=1$. Further every pair of points in $\Sigma$ is connected by a 
minimal geodesic. 
\end{lemma}

\begin{proof}
The first part follows directly from Clairaut's integral for the geodesic flow of $(\Sigma,k)$. For the second part note that one has 
$$\dist\nolimits^k(\{x=x_0\},\{x=x_1\})=\left|\int_{x_0}^{x_1} \sqrt{1+(r')^2}\,dx\right|\ge |x_1-x_0|$$
and 
$$L^k(\{x=x_0\})=2\pi r(x_0),$$
where $\dist^k$ denotes the distance and $L^k$ denotes the length relative to $k$. 
For $x_0$ sufficiently close to $0$ or $1$ one has thus  $L^k(\{x=x_0\})<1$, which implies that 
$$\dist\nolimits^k(\{x=x_0\},\{x=0,1\})>L^k(\{x=x_0\}).$$
Therefore no minimal geodesic between two points in $\Sigma$ intersects the singularities $\{x=0\}$ or $\{x=1\}$.
\end{proof}

\begin{cor}\label{cor3}
The spacetime 
$$(M,g)=(\R\times \Sigma, -dt^2+k)$$
is causally simple.
\end{cor}

\begin{proof}
Since the projection $M\to \R$ onto the first factor is a temporal function, $(M,g)$ is casual. It remains to show that $J^+_g$ is closed: One has 
\begin{equation}\label{E1}
((\sigma,p),(\tau,q))\in J^+_g\Leftrightarrow \tau-\sigma \ge \dist\nolimits^k(p,q)
\end{equation}
by Lemma \ref{L10}. The right-hand-side of \eqref{E1} is a closed condition. The closeness of $J^+_g$ follows directly.
\end{proof}

\begin{prop}\label{P10}
The space of null geodesics $\mathcal{N}_g$ is not Hausdorff.
\end{prop}

\begin{proof}
Up to parametrization every null geodesic $\gamma$ of $(M,g)$ is of the form $t\mapsto (t,\eta(t))$, where $\eta$ is a $k$-arclength geodesic. Choose a sequence 
$\{\eta_n\}_{n\in\N}$ of complete $k$-arclength geodesics whose tangents approach the meridian tangents $\frac{1}{\sqrt{1+(r')^2}}\partial_x$. The sequence $\{\eta_n\}_n$
then converges locally in every $C^l$-topology to a union of meridians of $(\Sigma,k)$. The induced sequence $\{\gamma_n\}_n$ has thus 
several limits in the space of null geodesics, i.e. $\mathcal{N}_g$ is not Hausdorff.
\end{proof}

Theorem \ref{T2} follows from Corollary \ref{cor3} and Proposition \ref{P10} in conjunction with Theorem \ref{cor1}.

\begin{appendix}
\section{The index form of a null geodesic}\label{A1}

%
%
%
%
%

Here we recall the definition and the properties of the  index form of a null geodesic for the convenience of the reader. The material with proofs and 
additional explanations can be found in \cite[Chapter 10]{Beem} and the references therein. Let $(M,g)$ be a spacetime and $\gamma\colon [a,b]\to M$ 
be a null geodesic. Let 
$$\gamma^\perp:=\{v\in \gamma^*TM|\; g(v,\dot\gamma)=0\}$$
denote the orthogonal bundle to $\gamma$. Define a equivalence relation $\sim$ on $\gamma^\perp$ by setting $v\sim w$ if 
$w-v\in \text{span}(\dot\gamma)$. Denote with $\overline{v}$ the equivalence class of $v\in \gamma^\perp$. The quotient bundle 
$$\overline{\gamma^\perp}:=\gamma^\perp/\sim$$ 
is a smooth bundle over $[a,b]$. Since $\gamma$ is null one has $g(v_1,w_1)=g(v_2,w_2)$ for all $v_1,v_2,w_1,w_2\in \gamma^\perp$ with 
$\overline{v_1}=\overline{v_2}$ and $\overline{w_1}=\overline{w_2}$. The same is true for the curvature endomorphism $R(.,\dot\gamma)\dot\gamma$, i.e. 
$$R(v,\dot\gamma)\dot\gamma=R(w,\dot\gamma)\dot\gamma\in \gamma^\perp$$
if $v-w\in \text{span}(\dot\gamma)$. Thus both the metric $g$ and the curvature endomorphism $R(.,\dot\gamma)\dot\gamma$ descend 
to a well defined metric $\overline{g}$ on $\overline{\gamma^\perp}$ with
$$\overline{g}(\overline{v},\overline{w}):=g(v,w)$$
and a well defined endomorphism field $\overline{R}(.,\dot\gamma)\dot\gamma$ on $\overline{\gamma^\perp}$ with
$$\overline{R}(\overline{v},\dot\gamma)\dot\gamma):=\overline{R(v,\dot\gamma)\dot\gamma}.$$
If $X\in \Gamma(\gamma^\perp)$ then the covariant derivative is again a smooth section of $\gamma^\perp$ and if $X-Y\in \text{span}(\dot\gamma)$ 
everywhere, then $\nabla_{\dot\gamma}(X-Y)\in \text{span}(\dot\gamma)$ everywhere as well. Therefore the covariant derivative $\nabla_{\dot\gamma}$ 
descends to a covariant derivative on $\overline{\gamma^\perp}$. Abbreviate the covariant derivative by a prime, i.e. 
$$V':=\overline{\nabla_{\dot\gamma} X},$$
where $V$ denotes the quotient section of $X\in \Gamma(\gamma^\perp)$, i.e. $\overline{X}_t=V_t$ for all $t\in [a,b]$.

Denote with $\mathfrak{X}(\gamma)$ the piecewise smooth sections of $\overline{\gamma^\perp}$ and let 
$$\mathfrak{X}_0(\gamma):=\{V\in \mathfrak{X}(\gamma)|\; V_a=0_a\text{ and }V_b=0_b\},$$
where $0_t$ denotes the zero vector in $\overline{\gamma^\perp}_t$.

\begin{definition}
A smooth section $V\in \mathfrak{X}(\gamma)$ is said to be a {\it Jacobi class} in $\gamma^\perp$ if $V$ satisfies the Jacobi equation 
$$V''+\overline{R}(V,\dot\gamma)\dot\gamma=0.$$
\end{definition}

\begin{lemma}
Let $W$ be a Jacobi class in $\mathfrak{X}(\gamma)$. Then there exists a Jacobi field $Y\in \Gamma(\gamma^\perp)$ with $\overline{Y_t}=W_t$
for all $t\in [a,b]$. Conversely if $Y$ is a Jacobi field in $\gamma^\perp$, then $t\mapsto W_t:=\overline{Y_t}$ is a Jacobi class in $\overline{\gamma^\perp}$.
\end{lemma}

\begin{lemma}
Let $W\in \mathfrak{X}(\gamma)$ be a Jacobi class with $W_a=0_a$ and $W_b=0_b$. Then there is a unique 
Jacobi field $Z\in \Gamma(\gamma^\perp)$ with $\overline{Z_t}=W_t$ for all $t\in [a,b]$ that vanishes at $a$ and $b$.
\end{lemma}

\begin{definition}
For $s\neq t\in [a,b]$, $s$ and $t$ are said to be {\it conjugated along $\gamma$} if there exists a Jacobi class $W\neq 0$ in 
$\mathfrak{X}(\gamma)$ with $W_{s}=0_s$ and $W_{t}=0_t$. Also $t\in (a,b]$ is said to be a {\it conjugate point of $\gamma$}
if $s=a$ and $t$ are conjugate along $\gamma$. 
\end{definition}

\begin{definition}
The {\it index form} $\overline{I}\colon \mathfrak{X}(\gamma)\times \mathfrak{X}(\gamma)\to \R$ is given by
$$\overline{I}(V,W)=-\int_a^b   [\overline{g}(V',W')-\overline{g}(\overline{R}(V,\dot\gamma)\dot\gamma,W)]dt.$$
\end{definition}

\begin{theorem}
Let $\gamma\colon [a,b]\to M$ be a null geodesic segment. Then the following are equivalent:
\begin{itemize}
\item[(a)] The segment $\gamma$ has no conjugate points to $s=a$ in $(a,b]$.
\item[(b)] $\overline{I}(W,W)<0$ for all $W\in \mathfrak{X}_0(\gamma)$, $W\neq 0$.
\end{itemize}
\end{theorem}

\end{appendix}


\end{document}